\documentclass[a4paper,reqno]{amsart}

\usepackage{amssymb}
\usepackage{enumerate}

\usepackage{hyperref}
\hypersetup{
    colorlinks=true,
    linkcolor=blue,
    citecolor=blue,
    urlcolor=blue,
}

\newtheorem{theorem}{Theorem}
\newtheorem{corollary}[theorem]{Corollary}
\newtheorem{lemma}[theorem]{Lemma}
\newtheorem{proposition}[theorem]{Proposition}

\theoremstyle{remark}
\newtheorem*{remark}{Remark}
\newtheorem*{remarks}{Remarks}

\numberwithin{equation}{section}
\numberwithin{theorem}{section}

\newcommand{\nfrac}[3]{\frac{#2}{\raisebox{-#1}{$#3$}}}
\newcommand{\fracn}[3]{\frac{\raisebox{-#1}{$#2$}}{#3}}

\def\half{\fracn{1pt}{1}{2}}

\def\N{{\mathbb N}}

\def\R{{\mathbb R}}

\def\Z{{\mathbb Z}}

\def\B{{\mathcal B}}

\def\F{{\mathcal F}}

\def\G{{\mathcal G}}

\def\g{\widetilde g}

\def\L{{\mathcal L}}

\def\T{T_+}

\def\X{{X_+}}

\def\1{{\bf 1}}
\def\var{{\rm var}}

\begin{document}

\title{Some comments on `Unique Bernoulli $g$-measures'}

\author{Paul Hulse}

\thanks{Email: phulse@ed.ac.uk}

\begin{abstract}
Proofs of results due to Johansson, \"Oberg and Pollicott \cite{JOP}
are given which correct a bound used in the original.
This leads to modifications to the most general results;
however, the main corollaries are unaffected.
\end{abstract}

\maketitle

\section{Introduction}
\label{intro}

In what follows, it is assumed that a finite set $S$ has the discrete topology,
and if $G$ is a countable set, the sequence space $Y=S^G$
has the product topology,
with respect to which it is compact.
The set of Borel probability measures on $Y$ is denoted by $M(Y)$,
and $B(Y)$ is the set of bounded real-valued Borel functions on $Y$.
If $\zeta\in S^\Lambda$ ($\Lambda\subseteq G$) and $K\subseteq\Lambda$,
then $\zeta_K$ denotes the natural projection of $\zeta$ onto $S^K$,
$[\zeta]$ is the cylinder set
$\{y\in Y\,:\,y_\Lambda=\zeta\}$,
$[\zeta]_K=[\zeta_K]$,
and the sub-$\sigma$-algebra generated
by $\{[\zeta]\,:\,\zeta\in S^\Lambda\}$ is denoted by $\B_\Lambda$
(or $\B^Y_\Lambda$ if we wish to emphasise the space);
if $\Lambda\cap\Lambda^\prime=\emptyset$ and $\eta\in S^{\Lambda^\prime}$,
then $\zeta\eta\in S^{\Lambda\cup\Lambda^\prime}$ is defined by
\begin{equation*}
(\zeta\eta)_i=
\begin{cases}
\zeta_i,& i\in\Lambda,\\
\eta_i,& i\in\Lambda^\prime.
\end{cases}
\end{equation*}
The variation of $f\in B(Y)$ is measured by
$$
\var_\Lambda(f)=\sup\{f(x)-f(y)\,\colon x_\Lambda=y_\Lambda\}
\qquad(\Lambda\subseteq G),
$$
and for positive $f\in B(Y)$,
$$
\rho_\Lambda(f)=\sup\{f(x)/f(y)\,\colon x_\Lambda=y_\Lambda\}.
$$
(Note that $\log \rho_\Lambda(f)=\var_\Lambda(\log f)$.)

Throughout, $X$ denotes $S^\Z$ (for some finite set $S$ with $|S|\geq2$),
and $\X$ denotes $S^{\Z^+}$ (where $\Z^+$ is the set of non-negative integers;
$\Z^-$ denotes the non-positive integers);
$T$ and $\T$ denote the left shifts on $X$ and $\X$, respectively.
Let
$$
\G=\biggl\{g\in B(\X)\,:\,g\geq0,\ \sum_{y\in \T^{-1}x}g(y)=1
\ \forall\, x\in\X\biggr\}.
$$
If $g\in\G$, a measure $\mu\in M(X)$ is said to be a $g$-chain if
\begin{equation}\label{eq:gmeas}
\mu\bigl([x]_{\{n\}}\bigm|\B_{(n,\infty)}\bigr)(x)
=\g(T^nx)\qquad\hbox{a.e.}(\mu)
\end{equation}
for all $n\in\Z$
(here and later, $\g$ denotes the natural extension of $g$ to $X$,
and we apply the usual interval notation to subsets of $\Z$).
Equivalently, \eqref{eq:gmeas} defines $g$-chains on $\X$
(with the obvious modifications);
the natural projection from $X$ onto $\X$ determines
a one-to-one correspondence between the two.
A $T$- or $\T$-invariant $g$-chain is usually referred to as a $g$-measure;
if $g$ is positive and continuous, there is always at least one $g$-measure.

In \cite{JOP}, Johansson, \"Oberg and Pollicott obtained sufficient conditions
on positive, continuous $g\in\G$ for there to be a unique $g$-chain,
which is $T$-invariant and Bernoulli, that is,
the dynamical system $(T,\mu)$ is isomorphic to a Bernoulli shift.
These conditions subsumed and extended many of the existing conditions for
uniqueness and Bernoullicity, specifically those in terms of $\var_{[0,n]}(g)$.
(For more on the background to these problems, see \cite{JOP}.)
In their most general form (Theorems 1.1, 2.2 and 2.5 of \cite{JOP}),
the hypotheses are in terms of both the variation of $g$ and a sequence
of natural numbers which together determine a suitable
{\em block-variation pair} (see \cite{JOP}),
but as corollaries, uniqueness and the Bernoulli property
are obtained in the following three cases:
\begin{enumerate}
\item
\hypertarget{hyp1}{
$$
\sum_{n=1}^\infty\bigl(\var_{[0,n]}(\log g)\bigr)^2<\infty;
$$}
\item
\hypertarget{hyp2}{for some $\varepsilon>0$,
$$
\sum_{n=0}^\infty\prod_{i=0}^n\rho_{[0,i]}(g)^{-(\frac{1}{2}+\varepsilon)}
=\infty;
$$}
\item
\hypertarget{hyp3}{
$$
\var_{[0,n]}(\log g)=o\left(n^{-\frac{1}{2}}\right)
\quad\text{as }n\to\infty.
$$}
\end{enumerate}
The method of proof can be summarised roughly as follows.
A block coupling (determined partly by the block-variation pair)
is used to couple extensions of a $g$-chain with different initial distributions,
and an asymptotic bound for the probability of the extensions
disagreeing at a coordinate is obtained using
Hellinger integral estimates and the Renewal Theorem.

However, it seems the bound obtained from the Renewal Theorem
needs to be modified.
The main effect of this change on the generality of the results
seems to be to restrict the rate of growth of the sequence which can
be used in the block-variation pair.
The hypotheses (1)--(3), however, are unaffected;
indeed (2) is covered by an earlier result in \cite{H}
(see Section \ref{sect:(2)}).

The modified version of Theorems 1.1, 2.2 and 2.5 of \cite{JOP}
is Theorem \ref{thm:g}, from which Theorem \ref{thm:c} follows,
and from which in turn, results with hypotheses (1) and (3) can be deduced.
The proofs are based entirely on the techniques of \cite{JOP};
the differences between Theorem \ref{thm:g}
and the results in \cite{JOP} arise from differences between Lemma \ref{lem:nu}
and \cite[Lemma 2.3]{JOP},
which are discussed in remarks at the end of Section \ref{sect:(3)}.

\section{More on \texorpdfstring{$g$}{g}-measures}\label{sect:(2)}

Let $g\in\G$ be continuous.
The operator
$\L_g\colon C(\X)\rightarrow C(\X)$ is defined by
$$
\L_g f(x)=\sum_{y\in\T^{-1}x}g(y)f(y)
\qquad(f\in C(\X)).
$$
A measure $\mu\in M(\X)$ is a $g$-measure if and only if $\mu$ is $\L_g$-invariant,
that is, $\mu(\L_gf)=\mu(f)$ $(f\in C(\X))$.
Any weak*-limit of measures of the form
$$
\mu_k(f)=n_k^{-1}\sum_{i=1}^{n_k}\L_g^if(x^{(k)})
$$
is a $g$-measure, while any weak*-limit of the form
$$
\nu_k(f)=\L_g^{n_k}f(x^{(k)})
\qquad(f\in C(\X))
$$
is a $g$-chain
(where $x^{(k)}\in \X$ and $n_k\nearrow\infty$).
Thus, if $g\in\G$ is continuous, there is always at least one $g$-measure,
and uniqueness is equivalent to the convergence (pointwise or uniform) of
$n^{-1}\sum_{i=0}^{n-1}\L_g^if$
to a constant for every $f\in C(\X)$,
whereas there is a unique $g$-chain if and only if
$\L_g^nf$ converges to a constant.

The following result is proved in \cite{H}.

\begin{theorem}\label{thm:h}
Let $g\in\G$, and let
$$
d_n=\sup\left\{\half\sum_{s\in S^{\{0\}}}
|g(sx_\N)-g(sy_\N)|\,:\,x_i=y_i,\,0\leq i\leq n-1\right\}
\qquad(n\geq1).
$$
If $\sum_{n=1}^\infty\prod_{i=1}^n(1-d_i)=\infty$,
there is a unique $g$-measure $\mu\in M(\X)$.
The natural extension of $\mu$ is Bernoulli,
and $\L_g^nf\to\mu(f)$ uniformly as $n\to\infty$, for all $f\in C(\X)$.
\end{theorem}

The validity of (\hyperlink{hyp2}{2}) as a hypothesis follows from the above theorem.
To see this, note that
$$
\inf_{s\in S^{\{0\}}}\frac{g(sx_\N)}{g(sy_\N)}
=\left(\sup_{s\in S^{\{0\}}}\frac{g(sy_\N)}{g(sx_\N)}\right)^{-1},
$$
and so if $x_i=y_i$ for $1\leq i\leq n$,
\begin{equation*}
\left|1-\frac{g(sx_\N)}{g(sy_\N)}\right|\leq
\begin{cases}
1-\rho_n^{-1},&\text{if }g(sx_\N)\leq g(sy_\N),\\
\rho_n-1,&\text{otherwise},
\end{cases}
\end{equation*}
where $\rho_n=\rho_{[0,n]}(g)$.
Since $1-\rho_n^{-1}\leq\rho_n-1$,
\begin{align*}
\sum_{s\in S^{\{0\}}}|g(sx_\N)-g(sy_\N)|
&= \sum_{s\in S^{\{0\}}}g(sy_\N)\left|1-\frac{g(sx_\N)}{g(sy_\N)}\right|\\
&\leq\rho_n-1.
\end{align*}
Moreover, if $g$ is positive, $d_i$ is uniformly bounded below 1, and so,
\begin{align*}
\prod_{i=k}^n(1-d_i)
&=\exp\left(\sum_{i=k}^n\log(1-d_i)\right)\\
&\geq\exp\left(\sum_{i=k}^n(-d_i-Cd_i^2)\right)\\
&=\exp\left(\sum_{i=k-1}^{n-1}
\left(-\tfrac{1}{2}(e^{\log\rho_i}-1)-C(e^{\log\rho_i}-1)^2\right)\right)\\
&\geq\exp\left(-\half\sum_{i=k-1}^{n-1}(\log\rho_i+K(\log\rho_i)^2)\right),
\end{align*}
for some constants $C$, $K$.
Given $\varepsilon>0$, we can choose $k$ so that
$K\log\rho_i<\varepsilon$ for all $i\geq k-1$,
in which case,
$$
\prod_{i=k}^n(1-d_i)\geq\exp\left(-\left(\tfrac{1}{2}+\varepsilon\right)
\sum_{i=k-1}^{n-1}\log\rho_i\right)
=\prod_{i=k-1}^{n-1}\rho_i^{-(\frac{1}{2}+\varepsilon)}.
$$
Thus,
$$
\sum_{n=0}^\infty\prod_{i=0}^n\rho_{[0,i]}(g)^{-(\frac{1}{2}+\varepsilon)}=\infty
\quad\Longrightarrow\quad\sum_{n=1}^\infty\prod_{i=1}^n(1-d_i)=\infty.
$$

\section{\texorpdfstring{$g$}{g}-measures and couplings}\label{sect:(3)}

In this section, we consider the block-couplings described in \cite{JOP}.
We adopt the convention that $[m,n]=\emptyset$ if $m>n$,
$[\zeta]_\emptyset$ and
empty intersections are the full space,
empty products are 1, and empty sums are 0.
If $s\in S$, then $s^\Lambda\in S^\Lambda$ is such that $s^\Lambda_i=s$
for all $i\in\Lambda$
($s^\emptyset$ is the empty sequence).
We identify $\prod_{i=1}^nS_i^\Z$
with the sequence space $(\prod_{i=1}^nS_i)^\Z$ in the obvious way,
and extend previous definitions and notation accordingly;
in particular, $[\zeta_1,\ldots,\zeta_n]_K=\prod_{i=1}^n[\zeta_i]_K$
$(\zeta_i\in S_i^\Lambda,\,K\subseteq\Lambda\subseteq\Z)$.
Additionally, $\Delta_\Lambda=\{(x,y)\in X\times X\colon x_\Lambda=y_\Lambda\}$.

Each $g\in\G$ determines,
for all $m\leq n\in\Z$,
a set of probability measures
$\{\pi_{[m,n]}^g(\,\cdot\,|\,x)\,:\,x\in X\}$ on $(X,\B_{[m,n]})$ defined by
\begin{equation}\label{eq:pi}
\pi_{[m,n]}^g\bigl([x^\prime]_{[m,n]}\,\big|\,x\bigr)=\prod_{i=m}^n\g_i(T^ix^\prime)
\qquad(x^\prime\in[x]_{(n,\infty)}).
\end{equation}
The $\pi_{[m,n]}^g(\,\cdot\,|\,x)$ are consistent in the sense that
$$
\pi_{[m,n]}^g\bigl([x]_{[m,n]}\,\big|\,x\bigr)
=\pi_{[m,i]}^g\bigl([x]_{[m,i]}\,\big|\,x\bigr)
\pi_{(i,n]}^g\bigl([x]_{(i,n]}\,\big|\,x\bigr)
\qquad(m\leq i<n),
$$
and measurable in the sense that $\pi_{[m,n]}^g(f\,|\,x)$ is measurable in $x$
for all $\B_{[m,n]}$-measurable $f$.
Thus, they can be extended to a measurable set
$\{\pi_{\Z^-}^g(\,\cdot\,|\,x)\,:\,x\in X\}$
of probability measures on $(X,\B_{\Z^-})$.
Note that
$$
\mu\bigl([x]_{[m,n]}\,\big|\,\B_{(n,\infty)}\bigr)(x)
=\pi_{[m,n]}^g\bigl([x]_{[m,n]}\,\big|\,x\bigr)
\qquad\text{a.e.}(\mu)
$$
for any $g$-chain $\mu$.
Note also that \eqref{eq:pi} sets up a 1-1 correspondence between
such measurable sets of probability measures and elements of $\G$.

If $(M_i,{\mathcal A}_i,\mu_i)$, $1\leq i\leq n$, are probability spaces,
a coupling of the measures $\{\mu_i\}_{i=1}^n$
is a probability measure $P$ on
$\prod_{i=1}^n(M_i,{\mathcal A}_i)$
such that
\[
P=\mu_i\circ\psi_i\qquad(1\leq i\leq n),
\]
where $\psi _i\colon \prod_{j=1}^nM_j\to M_i$ is the natural projection.
(Note we assume the product is ordered in the same way the measures are listed.)
For $\mu\in M(X)$  and $\zeta\in S^\Lambda$ $(\Lambda\subset\Z)$
such that $\mu([\zeta])>0$,
let $\mu_\zeta$ denote the conditional probability measure
$\mu(\,\cdot\,|[\zeta])$.
If $\mu\in M(X)$ is $T$-invariant, the dynamical system $(T,\mu)$ is Bernoulli
if (and only if) the process determined by the partitions
$\{[s^{\{i\}}]\colon s\in S\}$, $i\in\Z$, is very weak Bernoulli,
that is, given $\varepsilon>0$, there exist $n\in\N$,
subsets $G_m\subseteq S^{[1,m]}$ $(m\in\N)$,
and couplings $P_\zeta$ of $\bigl\{\mu,\,\mu_\zeta\bigr\}$
on $\B_{[-n,0]}\times\B_{[-n,0]}$ $(\zeta\in G_m)$ such that
$\sum_{\zeta\in G_m}\mu([\zeta])>1-\varepsilon$
and
$$
n^{-1}\sum_{i=0}^{n-1}P_\zeta\bigl(\Delta_{\{-i\}}^c\bigr)<\varepsilon
\qquad(\zeta\in G_m,\,m\in\N)
$$
(see \cite{O}).

Let $g_1,g_2\in\G$. By a coupling of $g_1$ and $g_2$,
we mean a set of probability measures $\{P_{x,y}\,:\,x,y\in X\}$
such that $P_{x,y}$ is a coupling of
$\{\pi^{g_1}_{[-n,0]}(\,\cdot\,|\,x),\,\pi^{g_2}_{[-n,0]}(\,\cdot\,|\,y)\}$,
and $P_{x,y}\bigl([\zeta,\eta]\bigr)$ is $\B_\N\times\B_\N$-measurable in $(x,y)$ for all
$\zeta,\eta\in S^{[-n,0]}$, $n\geq1$.
Note that for such a coupling,
\begin{equation}\label{eq:coupa}
\begin{aligned}
\bigl|\L_{g_1}^n\1[\zeta](x_{\Z^+})-\L_{g_2}^n\1[\zeta](y_{\Z^+})\bigr|
&=\bigl|\pi_{[-n,0]}^{g_1}\bigl(T^n[\zeta]\,\big|\,x\bigr)
-\pi_{[-n,0]}^{g_2}\bigl(T^n[\zeta]\,\big|\,y\bigr)\bigr|\\
&\leq P_{x,y}\bigl(\Delta^c_{[-n,m-n]}\bigr)
\end{aligned}
\end{equation}
for all $\zeta\in S^{[0,m]},\ 0\leq m\leq n,\ x,y\in X$.

Let $g\in\G$, $\mu\in M(X)$ be a $g$-measure,
and $\{P_{x,y}\}$ be a coupling of $g$ with itself.
Then for each $n\geq0$ and $\zeta\in S^{[1,m]}$,
the measure $P_\zeta$ on $(X\times X,\B_{[-n,0]}\times\B_{[-n,0]})$ defined by
$$
P_\zeta\bigl([\eta,\xi]\bigr)
=\iint P_{x,y}\bigl([\eta,\xi]\bigr)
\,d\mu(x)d\mu_\zeta(y)\qquad(\eta,\xi\in S^{[-n,0]})
$$
is a coupling of $\{\mu,\mu_\zeta\}$.
Thus, to show $(T,\mu)$ is Bernoulli, it is enough to show that
for any $\varepsilon>0$
there is a coupling $\{P_{x,y}\,:\,x,y\in X\}$ of $g$ with itself such that
\begin{equation}\label{eq:bern}
\liminf_{n\to\infty}\sup_{x,y\in X}
n^{-1}\sum_{i=0}^{n-1}P_{x,y}\bigl(\Delta^c_{\{-i\}}\bigr)<\varepsilon.
\end{equation}

Given probability measures $\mu$ and $\nu$ on $(X,\B_{[m,n]})$,
a coupling $P$ of $\bigl\{\mu,\nu\bigr\}$ can be defined by
\begin{equation}
P([\zeta,\zeta])=\min\{\mu([\zeta]),\nu([\zeta])\},
\label{eq:coup1}
\end{equation}
and
\begin{equation}P([\zeta,\eta])
=\frac{\bigl(\mu([\zeta])-P([\zeta,\zeta])\bigr)
\bigl(\nu\bigl([\eta]\bigr)-P\bigl([\eta,\eta]\bigr)\bigr)}
{P\bigl(\Delta_{[m,n]}^c\bigr)}
\qquad\bigl((\zeta,\eta)\in\Delta_{[m,n]}^c\bigr),
\label{eq:coup2}
\end{equation}
where $\zeta,\eta\in S^{[m,n]}$.
Note that
\begin{equation}\label{eq:coup3}
P\bigl(\Delta^c_{[m,n]}\bigr)=\half\sum_{\zeta\in S^{[m,n]}}
\bigl|\mu([\zeta])-\nu([\zeta])\bigr|.
\end{equation}

\begin{remark}
The righthand side of \eqref{eq:coup2} is taken to be 0
if $P\bigl(\Delta^c_{[m,n]}\bigr)=0$,
that is, if $\mu=\nu$;
the precise form of $P$ on $\Delta_{[m,n]}^c$ is not important,
only that a coupling satisfying \eqref{eq:coup1},
and hence \eqref{eq:coup3}, exists,
and is continuous in $\mu$, $\nu$.
\end{remark}

Let $Z=\{0,1\}^\Z$ and $\tau$ be the shift on $Z$.
From hereon, if $B=\{b_n\}_{n=1}^\infty\subseteq\N$,
then $B_n$ denotes $\sum_{i=1}^nb_i$ $(n\geq1)$, and $B_0=0$.
For each such sequence $B$ and $K\in\N$ and $z\in Z$,
a sequence $\{I_n(z)\}_{n=0}^\infty$ of intervals in $\Z$ is assigned,
with $I_0(z)=\emptyset$, and for $n\geq1$,
$I_n(z)=(a_n,a_{n-1}]$ defined inductively by $a_0=0$ and
\begin{equation}\label{eq:gcoup1}
a_{n}=
\begin{cases}
a_{n-1}-b_{k},&\text{if }
z\in\bigcup_{i=1}^{k-1}[0^{I_{n-i}(z)}]\setminus[0^{I_{n-k}(z)}],
\ 2\leq k\leq \min\{n,K+1\},\\
a_{n-1}-b_1,&\text{otherwise}.
\end{cases}
\end{equation}

Given $K\in\N$ and sequences $B=\{b_n\}_{n=1}^\infty\subseteq\N$,
$\{d_n\}_{n=1}^\infty\subseteq\R^+$,
a probability measure $\nu$ on $(Z,\B_{\Z^-})$ is defined as follows.
For each $n\geq1$ and
$z\in\bigcap_{j=1}^{k}[0^{I_{n-j}(z)}]\setminus[0^{I_{n-k-1}(z)}]$,
$0\leq k\leq n-1$, let
\begin{equation}\label{eq:gam1}
\nu\Biggl([z]_{I_n(z)}\Biggm|\bigcap_{j=1}^{n-1}[z]_{I_j(z)}\Biggr)
=\begin{cases}
1-d_{k+1},&z\in[0^{I_n(z)}],\ k\leq K-1,\\
d_{k+1},&z\in[1^{I_n(z)}],\ k\leq K-1,\\
1,&z\in[1^{I_n(z)}],\ k\geq K,\\
0,&\text{otherwise}.
\end{cases}
\end{equation}
(Recall that, for $n=1$, the empty intersection is $Z$.)
This determines $\nu$ on the partitions
$$
\left\{\bigcap_{k=1}^n[z]_{I_k(z)}\,:\,z\in Z\right\}
\qquad(n\geq1),
$$
and hence by extension, on $\B_{\Z^-}$.

\begin{lemma}\label{lem:nu}
Given $K\in\N$, and sequences $B=\{b_n\}_{n=1}^\infty\subseteq\N$ and
$\{d_n\}_{n=1}^\infty\subseteq\R^+$,
let $\nu$ be the measure defined by \eqref{eq:gcoup1}--\eqref{eq:gam1}.
Then
\begin{equation}\label{eq:lem}
\lim_{n\to\infty}\nu\bigl([1]^{\{-n\}}\bigr)
=
\frac{\sum_{k=1}^{K}b_{k}d_k\prod_{j=1}^{k-1}(1-d_j)+b_{K+1}\prod_{j=1}^{K}(1-d_j)}
{\sum_{k=1}^{K+1} b_{k}\prod_{j=1}^{k-1}(1-d_j)}.
\end{equation}
\end{lemma}

\begin{proof}
It follows from \eqref{eq:gcoup1}
that if $k\geq1$ and $z\notin[0^{I_{k}(z)}]$,
then $I_{k+i}(z)=I_i(\tau^{a_k}z)+a_k$ $(i\geq1)$,
where $a_k$ is the righthand endpoint of $I_{k+1}(z)$,
and hence from \eqref{eq:gam1} that
$$
\nu\Bigl([z]_{[-n,a_k]}\,\Big|\,[z]_{(a_k,0]}\Bigr)
=\nu\bigl([\tau^{a_k}z]_{[-n-a_k,0]}\bigr)
\qquad(n\geq-a_k).
$$
In particular, if $i=B_{k}$ for some $0\leq k\leq K$, then
\begin{equation}\label{eq:pr1}
\nu\Bigl([1^{\{-n\}}]\,\Big|\,[1^{\{-i\}}0^{(-i,0]}]\Bigr)
=\begin{cases}
1,&i<n<B_{k+1},\\
\noalign{\vskip 4pt}
\nu\bigl([1^{\{B_{k+1}-n\}}]\bigr),& n\geq B_{k+1},
\end{cases}
\end{equation}
and
\begin{equation}\label{eq:pr2}
\nu\bigl([1^{\{-i\}}0^{(-i,0]}]\bigr)=
\begin{cases}
d_{k+1}\prod_{j=1}^{k}(1-d_j),& k\leq K-1,\\
\noalign{\vskip 8pt}
\prod_{j=1}^{K}(1-d_j),& k=K,
\end{cases}
\end{equation}
whereas if $i\neq B_k$ for any $0\leq k\leq K$, then
\begin{equation}\label{eq:pr3}
\nu\bigl([1^{\{-i\}}0^{(-i,0]}]\bigr)=0
\end{equation}
(note that in this case,  $1-i\leq0$).
It follows from \eqref{eq:pr1}--\eqref{eq:pr3} that for $B_k\leq n<B_{k+1}$, $k\geq0$,
\begin{align*}
\nu\bigl([1^{\{-n\}}]\bigr)
&=\sum_{i=0}^{n-1}
\nu\Bigl([1^{\{-n\}}]\,\Big|\,[1^{\{-i\}}0^{(-i,0]}]\Bigr)
\,\nu\bigl([1^{\{-i\}}0^{(-i,0]}]\bigr)\\
&\qquad+\nu\bigl([1^{\{-n\}}0^{(-n,0]}]\bigr)\\
&=\sum_{j=1}^{k}
\nu\bigl([1^{\{B_{j}-n\}}]\bigr)
\,\nu\bigl([1^{\{-B_{j-1}\}}0^{(-B_{j-1},0]}]\bigr)\\
&\qquad+\nu\bigl([1^{\{-B_k\}}0^{(-B_k,0]}]\bigr)
\end{align*}
(the cases $n=B_k$ and $n>B_k$ are different, but
lead to the same formula), and
\begin{equation*}
\nu\bigl([1^{\{-B_k\}}0^{(-B_k,0]}]\bigr)
=
\begin{cases}
d_{k+1}\prod_{i=1}^{k}(1-d_i),&0\leq k\leq K-1,\\
\noalign{\vskip 8pt}
\prod_{j=1}^{K}(1-d_j),&k=K,\\
\noalign{\vskip 8pt}
0,&k>K.
\end{cases}
\end{equation*}
Thus, if
\begin{equation*}
\alpha_i=
\begin{cases}
d_k\prod_{j=1}^{k-1}(1-d_j),&i=B_{k},\ 1\leq k\leq K,\\
\noalign{\vskip 8pt}
\prod_{j=1}^{K}(1-d_j),&i=B_{K+1},\\
\noalign{\vskip 4pt}
0,&\text{otherwise},
\end{cases}
\end{equation*}
and
\begin{equation}\label{eq:pr4}
\beta_n=
\begin{cases}
d_k\prod_{j=1}^{k-1}(1-d_j),& B_{k-1}\leq n<B_{k},\ 1\leq k\leq K,\\
\noalign{\vskip 8pt}
\prod_{j=1}^{K}(1-d_j),& B_K\leq n<B_{K+1},\\
\noalign{\vskip 4pt}
0,&\text{otherwise},\end{cases}
\end{equation}
then $\sum_{i=1}^\infty\alpha_i=1$,
and $\nu\bigl([1^{\{-n\}}]\bigr)$ satisfies the renewal equation
$$
\nu\bigl([1^{\{-n\}}]\bigr)
=\sum_{i=1}^n
\alpha_i\nu\bigl([1^{\{i-n\}}]\bigr)+\beta_n
\qquad(n\geq1),
$$
with $\nu\bigl([1^{\{0\}}]\bigr)=\beta_0$.
Let $m$ be the largest integer such that $B\subseteq m\N$,
or equivalently, such that $\{i\,:\,\alpha_i\neq0\}\subseteq m\N$.
Then
$\nu\bigl([1^{\{-mn\}}]\bigr)=\nu\bigl([1^{\{-mn-i\}}]\bigr)$
for all $0\leq i\leq m-1$, $n\geq0$,
and so it follows from the Renewal Theorem \cite[p330]{F} that
\begin{align*}
\lim_{n\to\infty}\nu\bigl([1^{\{-n\}}]\bigr)
&=\lim_{n\to\infty}\nu\bigl([1^{\{-mn\}}]\bigr)\\
&=\sum_{n=0}^\infty\beta_{mn}\left(\sum_{n=1}^\infty n\alpha_{mn}\right)^{-1}.
\end{align*}
Since $\beta_{mn+i}=\beta_{mn}$ for $0\leq i\leq m-1$,
$$
\sum_{n=0}^\infty\beta_{mn}
=m^{-1}\left(\sum_{k=1}^{K}b_{k}d_k\prod_{j=1}^{k-1}(1-d_j)
+b_{K+1}\prod_{j=1}^{K}(1-d_j)\right),
$$
and since $\alpha_{mn}=0$ unless $mn=B_k$ for some $0\leq k\leq K+1$,
\begin{align*}
\sum_{n=1}^\infty n\alpha_{mn}
&=\sum_{k=1}^{K}m^{-1}B_{k}d_k\prod_{j=1}^{k-1}(1-d_j)
+m^{-1}B_{K+1}\prod_{j=1}^{K}(1-d_j)\\
&=m^{-1}\left(\sum_{k=1}^{K+1} b_{k}\prod_{j=1}^{k-1}(1-d_j)\right).
\end{align*}
Thus, the result follows.
\end{proof}

Given $g\in\G$ and a sequence $B=\{b_n\}_{n=1}^\infty\subseteq\N$,
let
\begin{equation}\label{eq:dg}
d_n(g,B)=\sup\left\{
\half\sum_{\zeta\in S^{J_n}}
\bigl|\pi^g_{J_n}([\zeta]\,|\,x)-\pi^g_{J_n}([\zeta]\,|\,y)\bigr|
\,:\,(x,y)\in\Delta_{(-B_{n-1},0]}\right\},
\end{equation}
where $J_n=(-B_{n},-B_{n-1}]$ $(n\geq1)$,
and given $K\in\N$,
let $\nu$ be the  measure defined by \eqref{eq:gcoup1}--\eqref{eq:gam1}
with $d_n=d_n(g,B)$.
Couplings $P_{x,y}$
of $\bigl\{\pi^g_{\Z^-}(\,\cdot\,|\,x),\,\pi^g_{\Z^-}(\,\cdot\,|\,y),\,\nu\bigr\}$ $(x,y\in X)$
such that
\begin{equation}\label{eq:bc4}
P_{x,y}\bigl(\Delta^c_{\{-n\}}\times [0^{\{-n\}}]\bigr)=0\qquad(n\geq0),
\end{equation}
$P_{x,y}=P_{x^\prime,y^\prime}$ for $(x,y)_\N=(x^\prime,y^\prime)_\N$,
and $P_{x,y}\bigl([\zeta,\eta,\xi]\bigr)$ is continuous in $(x,y)$
for all $\zeta,\eta\in S^{[-n,0]},\,\xi\in\{0,1\}^{[-n,0]}$.
are defined as follows.

Let  $\F_n$ denote the sub-$\sigma$-algebra generated by the sets
$$
\left\{\bigcap_{k=1}^n[x,y,z]_{I_k(z)}\,:\,x,y\in X,\,z\in Z\right\}
\qquad(n\geq1)
$$
For an interval $I\subset\Z$, let $P_I(\,\cdot\,|\,x,y)$ denote the coupling of
$\bigl\{\pi^g_I(\,\cdot\,|\,x),\,\pi^g_I(\,\cdot\,|\,y)\bigr\}$ as determined by
\eqref{eq:coup1}--\eqref{eq:coup2},
and $\nu_I(\,\cdot\,|\,z)$ denote the restriction of
$\nu\bigl(\,\cdot\,\big|\,[z]_{(\max I,0]}\bigr)$ to $\B^Z_I$.
Note that if $(x,y)\in\bigcap_{j=1}^{k}\Delta_{I_{n-j}(z)}$
for some $z\in Z$, $0\leq k\leq n-1$,
then it follows from \eqref{eq:coup3} and \eqref{eq:dg} that
\begin{equation}\label{eq:dg2}
P_{I_n(z)}\bigl(\Delta_{I_n(z)}^c\bigm| x,y\bigr)\leq d_{k+1}.
\end{equation}
Fix $x^\prime,y^\prime\in X$, and suppose that $P_{x^\prime,y^\prime}$ has been
defined on  $\F_{n-1}$ for some $n\geq1$ so that
\begin{equation}\label{eq:bc1}
P_{x^\prime,y^\prime}\bigl(\Delta^c_{I_k(z)}\times [0^{I_k}]\bigr)=0
\end{equation}
for all $1\leq k\leq n-1$.
(When $n=1$, $\F_0$ is taken to be the trivial $\sigma$-algebra.)
Then it follows from \eqref{eq:dg2}, \eqref{eq:bc1} and \eqref{eq:gam1} that
\begin{equation}\label{eq:bc2}
P_{I_n(z)}\bigl(\Delta_{I_n(z)}^c\bigm| x,y\bigr)
\leq \nu_{I_n(z)}\big([1^{I_n(z)}]\bigm| z\bigr)
\qquad\text{a.e.}(P_{x^\prime,y^\prime}).
\end{equation}
Thus, for each $x,y\in X$ and $z\in Z$,
$P_{x^\prime,y^\prime}\bigl(\,\cdot\,\big|\,\bigcap_{k=1}^{n-1}[x,y,z]_{I_k(z)}\bigr)$
can be defined on $\B^X_{I_n(z)}\times\B^X_{I_n(z)}\times\B^Z_{I_n(z)}$ as a coupling of
$\bigl\{P_{I_n(z)}(\,\cdot\,| x,y),\,\nu_{I_n(z)}(\,\cdot\,|\,z)\bigr\}$
such that
\begin{equation}\label{eq:bc3}
P_{x^\prime,y^\prime}\Bigg(\Delta^c_{I_n(z)}\times[0^{I_n(z)}]\Biggm|
\bigcap_{k=1}^{n-1}[x,y,z]_{I_k(z)}\Bigg)=0
\qquad\text{a.e.}(P_{x^\prime,y^\prime}),
\end{equation}
and so that
$P_{x^\prime,y^\prime}\bigl([x,y,z]_{I_n(z)}\,\big|\,\bigcap_{k=1}^{n-1}[x,y,z]_{I_k(z)}\bigr)$
is continuous  in $(x^\prime,y^\prime)$
for all $x,y\in X$, $z\in Z$.
When $n=1$, \eqref{eq:bc2} holds and \eqref{eq:bc1} is trivial,
so it follows inductively that \eqref{eq:bc3} holds for all $n\geq1$.
This determines $P_{x,y}$ $(x,y\in X)$ on $\F_n$, and by extension as a coupling
of $\bigl\{\pi^g_{\Z^-}(\,\cdot\,|\,x),\,\pi^g_{\Z^-}(\,\cdot\,|\,y),\,\nu\bigr\}$,
and it follows from \eqref{eq:bc3} that \eqref{eq:bc4} holds.

The next result follows immediately from \eqref{eq:bc4} and Lemma \ref{lem:nu}.

\begin{proposition}\label{prop:gm}
Let $g\in\G$ be continuous, $B=\{b_n\}_{n=1}^\infty\subseteq\N$, $K\in\N$,
and $\{P_{x,y}\,:\,x,y\in X\}$ be coupling defined above by \eqref{eq:bc1}--\eqref{eq:bc3}.
Then
\begin{equation*}%\label{eq:prop}
\limsup_{n\to\infty}\sup_{x,y\in X}P_{x,y}\bigl(\Delta^c_{\{-n\}}\times Z\bigr)\\
\leq
\frac{\sum_{k=1}^{K}b_{k}d_k\prod_{j=1}^{k-1}(1-d_j)+b_{K+1}\prod_{j=1}^{K}(1-d_j)}
{\sum_{k=1}^{K+1} b_{k}\prod_{j=1}^{k-1}(1-d_j)},
\end{equation*}
where $d_n=d_n(g,B)$.
\end{proposition}

\begin{remarks}
In \cite{JOP}, the required asymptotic bound on the probability of two extensions
of a $g$-chain disagreeing at a coordinate is obtained by combining
the inequality \cite[(2.7)]{JOP} with \cite[Lemma 2.3]{JOP}, the aim of which is to show that
\begin{equation}\label{eq:jop}
\lim_{n\to\infty}Prob(Y_n\leq0)
\leq\nfrac{1pt}{1+\sum_{k=1}^Kb_k\exp(-\sum_{j=1}^{k-1}r_j)(1-e^{-r_k})}
{\sum_{k=1}^Kb_k\exp(-\sum_{j=1}^{k-1}r_j)},
\end{equation}
where $\{Y_n\}$ is the Markov chain  defined by \cite[(2.6)]{JOP}
in terms of sequences $B=\{b_n\}$ and $\{r_n\}$, and an arbitrary $K>0$.
The proof of \eqref{eq:jop} actually claims equality,
since $Prob(Y_n\leq0)$ satisfies a renewal equation
(see \cite[(3.5),(3.7),(3.8)]{JOP}),
and this is equivalent to calculating the limit of $\nu\bigl([1^{\{-n\}}]\bigr)$
as in Lemma \ref{lem:nu} above,
but with $d_k$ replaced by $1-e^{-r_k}$.
%Note, however, that the limits in \eqref{eq:lem} and \eqref{eq:jop} are different.
%here is no term in \eqref{eq:jop} that matches
%$b_{K+1}\prod_{j=1}^{K}(1-d_j)$ in the numerator of \eqref{eq:prop}.
%This term derives from the form of $\beta_n$ in \eqref{eq:pr4} for $B_K\leq n<B_{K+1}$,
%which follows ultimately from \eqref{eq:pr2} when $k=K$.
%as is the case with $\nu\bigl([1^{\{-n\}}]\bigr)$.
However, although in applications the sequences $\{b_n\}$ and $\{r_n\}$ would be
such that the righthand side of \eqref{eq:jop}
could be made arbitrarily small for large $K$,
the renewal equation holds for any choice of positive $r_n$;
these could be chosen so that $1-e^{-r_n}$ is uniformly close to 1,
in which case the righthand side of \eqref{eq:jop} could be greater than 1.

The difference between the bounds in \eqref{eq:jop} and \eqref{eq:lem} is
$$
\left|\nfrac{1pt}{1-(1-d_K)b_K\prod_{j=1}^{K-1}(1-d_j)}
{\sum_{k=1}^Kb_k\prod_{j=1}^{k-1}(1-d_j)}\right|
=\left|\nfrac{1pt}{1-e^{-r_K}b_K\exp(-\sum_{j=1}^{K-1}r_j)}
{\sum_{k=1}^Kb_k\exp(-\sum_{j=1}^{k-1}r_j)}\right|.
$$
This does not matter if the difference tends to 0 as $K\to\infty$,
but for small $r_n$ and $d_n$, \eqref{eq:jop} may be less than \eqref{eq:lem},
and so it could potentially be significant if, for example, $b_n$ grows exponentially as $n\to\infty$,
as is the case in the verification of hypotheses (\hyperlink{hyp1}{1}) and (\hyperlink{hyp3}{3}).
\end{remarks}

\section{The main results}

\begin{theorem}\label{thm:g}
Let $g\in\G$ be continuous.
If, given $\varepsilon>0$, there exists a sequence
$B=\{b_n\}_{n=1}^\infty\subseteq\N$
such that
\begin{equation}\label{eq:g}
\limsup_{n\to\infty}\frac
{\sum_{k=1}^nb_{k}d_k\prod_{j=1}^{k-1}(1-d_j)+b_{n+1}\prod_{j=1}^n(1-d_j)}
{\sum_{k=1}^{n+1} b_{k}\prod_{j=1}^{k-1}(1-d_j)}
<\varepsilon,
\end{equation}
where $d_n=d_n(g,B)$,
then there is a unique $g$-chain $\mu\in M(X)$, which is $T$-invariant
and Bernoulli.
\end{theorem}

\begin{proof}
Given $\varepsilon>0$,
we can choose $B=\{b_n\}_{n=1}^\infty$ so that
$$
\frac{\sum_{k=1}^{K}b_{k}d_k\prod_{j=1}^{k-1}(1-d_j)+b_{K+1}\prod_{j=1}^{K}(1-d_j)}
{\sum_{k=1}^{K+1} b_{k}\prod_{j=1}^{k-1}(1-d_j)}
<\varepsilon
$$
for some $K\in\N$.
Then if $\{P^\prime_{x,y}\,:\,x,y\in X\}$ is the coupling of $g$ with itself
defined by
$P^\prime_{x,y}=P_{x,y}\circ\psi^{-1}$ where
$\{P_{x,y}\,:\,x,y\in X\}$ is the coupling determined by \eqref{eq:bc1}--\eqref{eq:bc3},
and $\psi\colon X\times X\times Z\to X\times X$ is the natural projection,
the result follows from Proposition \ref{prop:gm} and \eqref{eq:bern}.
\end{proof}

\begin{remark}
Theorem \ref{thm:h} follows from Theorem \ref{thm:g} by taking $b_n=1$ for all $n$
(in which case the numerator in \eqref{eq:g} is 1).
\end{remark}

The proofs of the following lemmas are based on
the use of Hellinger integral estimates in \cite{JOP}.

\begin{lemma}\label{lem:rg1a}
Let $\Omega$ be a finite set, and $\mu,\nu\in M(\Omega)$
be such that $\mu(\omega),\nu(\omega)>0$ for all $\omega\in\Omega$.
Then
\begin{equation*}
\sum_{\omega\in\Omega}\sqrt{\mu(\omega)\nu(\omega)}
\geq 1-\tfrac{1}{2}\bigl(\sqrt{\rho}-1\bigr)^2,
\end{equation*}
where
$$
\rho=\sup_{\omega\in\Omega}\left\{\frac{\mu(\omega)}{\nu(\omega)},
\frac{\nu(\omega)}{\mu(\omega)}\right\}.
$$
\end{lemma}

\begin{proof}
For any $\omega\in\Omega$,
\begin{equation*}
2\sqrt{\mu(\omega)\nu(\omega)}
=\mu(\omega)+\nu(\omega)-\nu(\omega)
\left(1-\sqrt{\frac{\mu(\omega)}{\nu(\omega)}}\right)^2.
\end{equation*}
Note that
$$
\inf_{\omega\in\Omega}\frac{\mu(\omega)}{\nu(\omega)}\geq\rho^{-1},
$$
and so,
\begin{equation*}
\left|\,1-\sqrt{\frac{\mu(\omega)}{\nu(\omega)}}\,\right|\leq
\begin{cases}
1-\sqrt{\rho^{-1}},&\nu(\omega)\geq\mu(\omega),\\
\sqrt{\rho}-1,&\nu(\omega)\leq\mu(\omega).
\end{cases}
\end{equation*}
Since $1-\sqrt{\rho^{-1}}\leq\sqrt{\rho}-1$,
it follows that
$$
\sum_{\omega\in\Omega}\sqrt{\mu(\omega)\nu(\omega)}
\geq 1-\tfrac{1}{2}\bigl(\sqrt{\rho}-1\bigr)^2,
$$
as required.
\end{proof}

To simplify the notation in the following result,
we identify a finite sequence $\zeta$
with the corresponding cylinder set $[\zeta]$.

\begin{lemma}\label{lem:rg2}
Let $\mu,\,\nu\in M(X,\B_{[m,n]})$ for some $m\leq n$.
Then
$$
\half\sum_{\zeta\in S^{[m,n]}}
\bigl|\mu(\zeta)-\nu(\zeta)\bigr|
\leq\sqrt{1-\prod_{i=m}^n
\left(1-\tfrac{1}{2}\bigl(\sqrt{\rho_i}-1\bigr)^2\right)^2}
$$
provided $\rho_i\leq\bigl(1+\sqrt{2}\bigr)^2$,
where
$$
\rho_i=\sup\left\{\frac{\mu(s\,|\,\eta)}
{\nu(s\,|\,\eta)},\,
\frac{\nu(s\,|\,\eta)}
{\mu(s\,|\,\eta)}
\,\colon s\in S^{\{i\}},\,\eta\in S^{[i+1,n]}\right\}
\qquad(m\leq i\leq n).
$$
\end{lemma}

\begin{proof}
An application of the Cauchy-Schwartz inequality shows that
\begin{align*}
\sum_{\zeta\in S^{[m,n]}}|\mu(\zeta)-\nu(\zeta)|
&=\sum_{\zeta\in S^{[m,n]}}
\left|\sqrt{\mu(\zeta)}-\sqrt{\nu(\zeta)}\right|
\left(\sqrt{\mu(\zeta)}+\sqrt{\nu(\zeta)}\right)\\
&\leq 2\,\sqrt{1-\left(\sum_{\zeta\in S^{[m,n]}}
\sqrt{\mu(\zeta)\nu(\zeta)}\right)^2}.
\end{align*}
It follows from Lemma \ref{lem:rg1a} that the result holds when $m=n$,
and applying this to the measures $\mu(\,\cdot\,|\,\eta)$,
$\nu(\,\cdot\,|\,\eta\bigr)$ ($\eta\in S^{[m+1,n]}$) when $m<n$ gives
\begin{align*}
\sum_{\zeta\in S^{[m,n]}}\sqrt{\mu(\zeta)\nu(\zeta)}
&=\sum_{\eta\in S^{[m+1,n]}}\sum_{s\in S^{\{m\}}}
\sqrt{\mu(s\,|\,\eta)\,\nu(s\,|\,\eta)
\,\mu(\eta)\,\nu(\eta)}\\
&\geq\left(1-\tfrac{1}{2}\bigl(\sqrt{\rho_m}-1\bigr)^2\right)
\sum_{\eta\in S^{[m+1,n]}}\sqrt{\,\mu(\eta)\nu(\eta)}.
\end{align*}
The result follows inductively.
\end{proof}

\begin{corollary}\label{cor:rg2}
Let $g\in\G$ be positive and continuous,
and $B=\{b_n\}_{n=1}^\infty\subseteq\N$.
Then for $n$ large enough,
$$
d_n(g,B)
\leq\sqrt{1-\prod_{i=B_{n-1}}^{B_{n}-1}
\left(1-\tfrac{1}{2}\bigl(\sqrt{\rho_i}-1\bigr)^2\right)^2},
$$
where $\rho_i=\rho_{[0,i]}(g)$.
\end{corollary}

\begin{proof}
The result is obtained by applying Lemma \ref{lem:rg2} to the measures
$\pi^g_{J_n}(\,\cdot\,|\,x)$ and
$\pi^g_{J_n}(\,\cdot\,|\,y)$
for $(x,y)\in\Delta_{(-B_{n-1},0]}$ (see \eqref{eq:dg}),
noting that $\rho_{[0,n]}(g)\leq(1+\sqrt{2})^2$ for $n$ large enough.
\end{proof}

\begin{lemma}\label{lem:as1}
Given $1<\lambda<(\sqrt{2}+1)^2$, there exists $K>0$ such that
if $1\leq \rho_i\leq\lambda$ $(1\leq i\leq n)$, then
$$
\prod_{i=1}^n\left(1-\tfrac{1}{2}\bigl(\sqrt{\rho_i}-1\bigr)^2\right)^2
\geq 1-\sum_{i=1}^n\left(\frac{(\log \rho_i)^2}{4}+K(\log \rho_i)^3\right).
$$
\end{lemma}

\begin{proof}
Given $\lambda$ and $\rho_i$ as above, let
$$
u_i=(\sqrt{\rho_i}-1)^2\qquad(1\leq i\leq n).
$$
Then $u_i\leq(\sqrt{\lambda}-1)^2<2$, and so,
\begin{equation}\label{eq:as2}
\begin{aligned}
\prod_{i=1}^n\left(1-\tfrac{1}{2}\bigl(\sqrt{\rho_i}-1\bigr)^2\right)^2
&=\exp\left(2\sum_{i=1}^n\log\left(1-\frac{u_i}{2}\right)\right)\\
&\geq\exp\left(-\sum_{i=1}^n
\bigl(u_i+A_1u_i^2\bigr)\right)\\
&\geq 1-\sum_{i=1}^n\left(u_i+A_2u_i^2\right),
\end{aligned}
\end{equation}
for some $A_1>0$ depending only on $\lambda$,
and $A_2>0$ depending only on $A_1$ and $\lambda$.
Since $\rho_i\leq\lambda$,
\begin{equation}\label{eq:as1}
\begin{aligned}
u_i&=1-2\exp\left(\tfrac{1}{2}\log \rho_i\right)
+\exp\left(\log \rho_i\right)\\
&\leq\frac{(\log \rho_i)^2}{4}+A_3(\log \rho_i)^3,
\end{aligned}
\end{equation}
for some $A_3>0$ depending only on $\lambda$.
Combining \eqref{eq:as1} and \eqref{eq:as2} gives
$$
\prod_{i=1}^n\left(1-\tfrac{1}{2}\bigl(\sqrt{\rho_i}-1\bigr)^2\right)^2
\geq 1-\sum_{i=1}^n\left(\frac{(\log \rho_i)^2}{4}+K(\log \rho_i)^3\right)
$$
for some $K>0$ depending only on $A_1,A_2, A_3$ and $\lambda$,
and hence, ultimately only on $\lambda$, as required.
\end{proof}

\begin{corollary}\label{cor:as}
Let $g\in\G$ be positive and continuous,
and $B=\{b_n\}_{n=1}^\infty\subseteq\N$.
If $\rho_{[0,B_N]}(g)<(\sqrt{2}+1)^2$, there exists $K>0$ such that
for all $n\geq N$,
$$
d_n(g,B)
\leq\sqrt{\sum_{i=B_n}^{B_{n+1}-1}
\left(\frac{(\log \rho_i)^2}{4}+K(\log \rho_i)^3\right)},
$$
where $\rho_i=\rho_{[0,i]}(g)$.
\end{corollary}

\begin{proof}
Since the $\rho_i$ are decreasing,
the result follows from Corollary \ref{cor:rg2}.
\end{proof}

\begin{theorem}\label{thm:c}
Let $g\in\G$ be positive and continuous.
If
\begin{equation}\label{eq:hyp5}
\lim_{n\to\infty}\sum_{i=\lceil\lambda^{n-1}\rceil}^{\lceil\lambda^n\rceil}
(\log \rho_{[0,i]}(g))^2=0
\end{equation}
for some $\lambda>1$,
then there is a unique $g$-chain $\mu\in M(X)$, which is $T$-invariant and Bernoulli.
\end{theorem}

\begin{proof}
Suppose $\lambda>1$ satisfies \eqref{eq:hyp5}, and consider $l>1$ and $a>0$.
Choose $m\in\N$ so that
$l^{m/2}\geq\max\{\lambda,a^{-1}\}$,
in which case,
$al^{m(n-1)}\geq\lambda^{n-1}$ $(n\geq2)$.
Now choose $k\in\N$ so that
$\lambda^k\geq\max\{a,l/\lambda\}$.
Then
$$
al^{mn}=a(l/\lambda)^{mn}\lambda^{mn}
\leq\lambda^{k+kmn+mn}\leq\lambda^{cn},
$$
where $c=k+km+m$.
Thus, since $l^{m(n-1)}\leq l^{j-1}<l^j\leq l^{mn}$ for $m(n-1)+1\leq j\leq mn$,
and $\lim_{n\to\infty}\log \rho_{[0,n]}(g)=0$,
\begin{equation}\label{eq:l}
\begin{aligned}
\limsup_{n\to\infty}
\sum_{i=\lceil al^{n-1}\rceil}^{\lceil al^n\rceil}
(\log \rho_{[0,i]}(g))^2
&\leq\limsup_{n\to\infty}
\sum_{i=\lceil al^{m(n-1)}\rceil}^{\lceil al^{mn}\rceil}
(\log \rho_{[0,i]}(g))^2\\
&\leq\limsup_{n\to\infty}
\sum_{i=\lceil\lambda^{n-1}\rceil}^{\lceil\lambda^{cn}\rceil}
(\log \rho_{[0,i]}(g))^2\\
&=\limsup_{n\to\infty}\sum_{j=n}^{n+c-1}
\sum_{i=\lceil\lambda^{j-1}\rceil}^{\lceil\lambda^{j}\rceil-1}
(\log \rho_{[0,i]}(g))^2\\
&=0.
\end{aligned}
\end{equation}
In particular, if $B=\{b_n\}_{n=1}^\infty$ is defined by
$B_n=\lceil l^n/(l-1)\rceil$ $(n\geq1)$, then
$\lfloor l^n\rfloor\leq b_n\leq\lceil l^n\rceil$ for all $n\geq2$, and
$$
\lim_{n\to\infty}\sum_{i=B_{n-1}}^{B_{n}-1}(\log \rho_{[0,i]}(g))^2=0.
$$
Hence, if $d_n=d_n(g,B)$,
it follows from Corollary \ref{cor:as} that
$\lim_{n\to\infty}d_n=0$.
Thus, for $n$ large enough, $1-d_n\geq l^{-1}$,
and so there is a constant $\varepsilon>0$ such that
$$
b_k\prod_{j=1}^{k-1}(1-d_j)\geq\varepsilon
$$
for all $k\geq1$.
Therefore,
$\sum_{k=1}^\infty b_{k}\prod_{j=1}^{k-1}(1-d_j)=\infty$, and so,
$$
\lim_{n\to\infty}\frac
{\sum_{k=1}^nb_{k}d_k\prod_{j=1}^{k-1}(1-d_j)}
{\sum_{k=1}^{n+1} b_{k}\prod_{j=1}^{k-1}(1-d_j)}
=0.
$$
Moreover,
\begin{align*}
\limsup_{n\to\infty}\frac
{b_{n+1}\prod_{j=1}^n(1-d_j)}
{\sum_{k=1}^{n+1} b_{k}\prod_{j=1}^{k-1}(1-d_j)}
&\leq\limsup_{n\to\infty}\frac{b_{n+1}}{B_{n+1}}\\
&=l-1.
\end{align*}
Thus, since $l>1$ is arbitrary,
the result follows from Theorem \ref{thm:g}.
\end{proof}

\begin{remark}
A consequence of \eqref{eq:l} is that
if \eqref{eq:hyp5} holds for one $\lambda>1$, it holds for all $\lambda>1$.
\end{remark}

Consider the hypotheses (\hyperlink{hyp1}{1}) and (\hyperlink{hyp3}{3}).
Clearly, (1) implies \eqref{eq:hyp5}.
Moreover, for $\lambda>1$,
$$
\lim_{n\to\infty}\sum_{i=\lceil\lambda^{n-1}\rceil}^{\lceil\lambda^n\rceil}i^{-1}
=\log\lambda,
$$
and so if $\log \rho_{[0,n]}=o(n^{-\frac{1}{2}})$,
then \eqref{eq:hyp5} holds.
Thus, both (\hyperlink{hyp1}{1}) and (\hyperlink{hyp3}{3}) imply uniqueness and the Bernoulli property.


\begin{thebibliography}{15}


\bibitem{F}  W. Feller.
{\it An Introduction to Probability Theory and its Applications\/},
Vol. I, 3rd ed. Wiley, 1968.

\bibitem{H} P. Hulse.
Correction to `A class of unique $g$-measures'.
{\it Ergod. Theory Dynam. Sys.\/} 26 (2006) 435-437.

\bibitem{JOP}  A. Johansson, A. \"Oberg, M. Pollicott.
Unique Bernoulli $g$-measures.
{\it J. Eur. Math. Soc.\/} 14 (2012) 1599-1615.

\bibitem{O} D.S. Ornstein.
{\it Ergodic Theory, Randomness and Dynamical Systems\/}.
Yale Univ. Press, 1974.

\end{thebibliography}
\end{document}